\documentclass[BCOR8mm,DIV14]{scrartcl}
\usepackage{umlaut}     
\usepackage{bbm}  
\usepackage[english]{babel}
\usepackage{amsmath,amssymb,amsthm}
\newtheorem{corollary}{Corollary}
\newtheorem{proposition}{Proposition}
\newtheorem{lemma}{Lemma}
\newtheorem{remark}{Remark}
\begin{document}
\title{ Two Strange Constructions in the Euclidean Plane } 
\author{VOLKER  THÜREY  }    
   \maketitle
  \centerline{ MSC-class: 51N20 }      
  \begin{abstract}
     \centerline{Abstract}     
     We present two   
     new constructions in the usual euclidean plane. 
     We  only  deal with 'Grecian Geometry', with this phrase we mean
	   elementary  geometry in the two-dimensional space $ \mathbbm{R}^{2}$ . 
	   We  describe and prove two propositions about  'projections'.  The proofs need only 
	   elementary analytical knowledge.
     \end{abstract}   
	 The reader may find the foundations and  assumptions of the  following propositiones in many books about 
	 plane geometry, for instance in  \cite{Goetz/Schweizer/Franke/Schoenwald}, p.1-29 . Or you can look in 
	  \cite{Hartshorne} ,    \cite{Hall/Szabo} ,    \cite{Kinder/Spengler} ,     \cite{Filler} , 
    \cite{Coxeter}, \cite{Roe} . \ See  also  \cite{Backman/Cromie}, p.224-234 . 
   \begin{proposition}   \rm   \label{proposition Euclid1} 
    Let us take \ $ \mathbbm{R}^{2} = \{ ({\tt x}|{\tt y}) | \ {\tt x},{\tt y} \in \mathbbm{R} \} $ ,
     \ the two-dimensional  euclidean plane, with the  horizontal {\tt x}-axis and  the vertical  {\tt y}-axis.
     Assume two parallel straight lines \ $  G_S \ { \rm and} \ G_T $. 
     Assume a third     line \ $ L $ , \ not parallel to \ $  G_S, G_T $, respectively, with the property 
     that  \ $ L $ \ does not meet the origin $(0|0)$.  
     The  intersection of \ $ L $ \ with \ $ G_S $ \ is called \ $ S = ({\tt x}_S|{\tt y}_S) $, \ \ 
     and the intersection of \ $ L $ \ with \ $ G_T $ \ is called \  $ T = ({\tt x}_T|{\tt y}_T) $ . \
    Note that, in the case that \ $G_S, G_T $ \ are distinct, the three points \ $(0|0)$,$S,T$ \ are not collinear. \       We can draw  two lines \ $ Z_S$ and $Z_T $, \ \ $ Z_S $ \ connects 
     the origin $(0|0)$ and $S$ , \ and \  $ Z_T $ \ connects $(0|0)$ and $T$ . \ \ $ Z_S$ \ and \ $Z_T $ \ 
      are distinct if   \ $G_S$ \ and \ $ G_T $ \ are distinct. \ \ Now we distinguish two cases
     { \sf (A) } and  { \sf (B) }, but note that they overlap. 
     \\   { \sf (A)}: \
      In the case that \ $  G_S \ { \rm and} \ G_T $ \  are not parallel to the horizontal {\tt x}-axis,  \ 
      we have two intersections \ $ a_S, a_T $ \ of \ $ G_S $ \ and \ $ G_T $, respectively, with the {\tt x}-axis.
    Then there is an unique point \ \  $  P_{hor} = ({\tt x}_{hor}|{\tt y}_{hor})$  on  \ $L$  , \   
        such that \ \  $ ({\tt x}_{hor}-a_S | {\tt y}_{hor}) \in Z_T , \ { \rm  and }  \
     ({\tt x}_{hor}-a_T | {\tt y}_{hor}) \in Z_S $. \\
        { \sf (B)}: \
      In the case that \ $  G_S \ { \rm and} \ G_T $ \  are not parallel to the vertical {\tt y}-axis,  \ 
      we have two intersections \ $ b_S, b_T $ \ of \ $ G_S $ \ and \ $ G_T $, respectively, with the {\tt y}-axis.
     Then there is an unique point \ \    $  P_{ver} = ({\tt x}_{ver}|{\tt y}_{ver})$  on  \ $L$  , \   
        such that \ \  $ ({\tt x}_{ver} | {\tt y}_{ver}-b_S) \in Z_T , \ { \rm  and }  \
     ({\tt x}_{ver} | {\tt y}_{ver}-b_T ) \in Z_S $ . \\
    \end{proposition}
      \quad  \\     
      Before reading the proof of the proposition you should take a look on Picture 1 .   
       \\  {  } \\    \\  \\  \\   
	   \setlength{\parindent}{0mm}
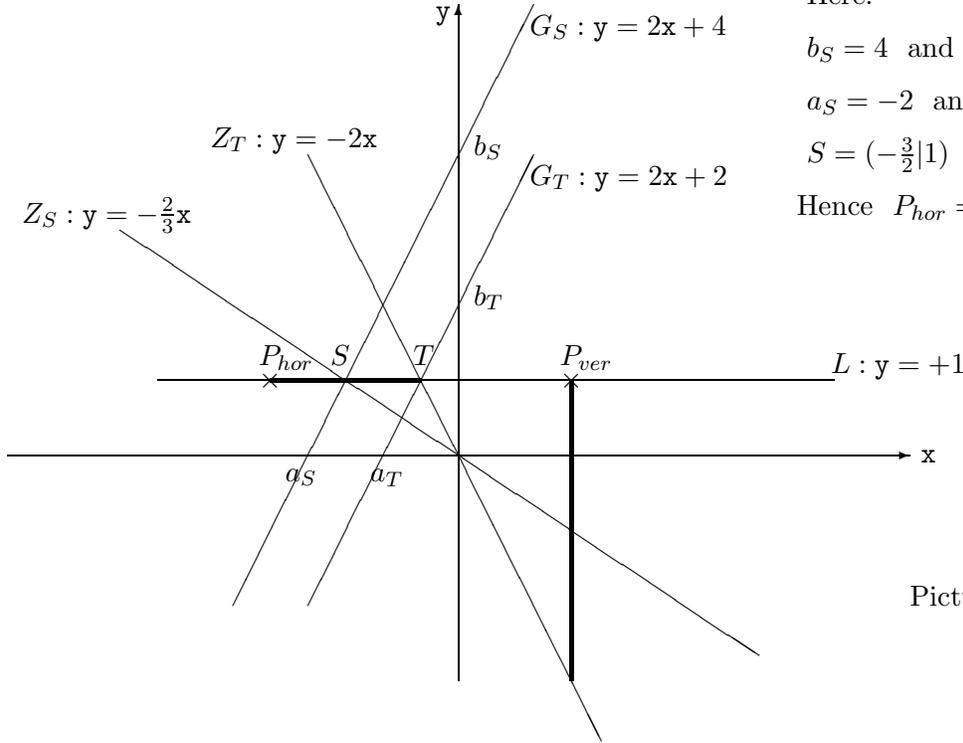
    \begin{figure}[ht]
    \centering
     \setlength{\unitlength}{1cm}
    \begin{picture}(4,6)     \thinlines      \put(6.0,-2){Picture 1} 
      \put(6.15,-0,1){$   {\tt x}   $}     \put(-0.3,5.8){$  {\tt y} $}    
       \put(-6.0,0){\vector(1,0){12.0}}     \put(0,-3){\vector(0,1){9}}  
       \put(-2.3,-0.34){$a_S$}     \put(0.95,3.6){$G_T: {\tt y}=2{\tt x}+2$}  
           \put(0.95,5.6){$G_S: {\tt y}=2{\tt x}+4$}   \put(4.95,1.1){$L: {\tt y}= +1$} 
       \put(-1.3,-0,35){ $ a_T $}  
       \put(-3.0,-2.0){\line(1,2){4.0}}  \put(-2.0,-2.0){\line(1,2){3.0}}   \put(-4.0,1.0){\line(1,0){9.0}} 
            \put(-1.7,1.2){$S$}     \put(-0.6,1.2){$T$}    \put(-2.0,4.0){\line(1,-2){3.9}}  
                                                           \put(-4.5,3.0){\line(3,-2){8.5}} 
          \put(-5.8,3.1){$Z_S: {\tt y}= -\frac{2}{3}{\tt x}$}    
         \put(-3.3,4.1){$Z_T:{\tt y}= -2{\tt x} $}  \put(-2.65,0.9){$\times$ }   \put(1.35,0.9){$\times$ } 
          \put(-2.65,1.2){$P_{hor}$ }    \put(1.35,1.2){$P_{ver}$ }  \put(0.2,2.0){$b_T$}    \put(0.2,4.0){$b_S$}  
      \put(4.5,6.0){ Here: }     
      \put(4.5,5.3){ $b_S = 4 $ \  and \ $b_T = 2 $ , } 
      \put(4.5,4.6){ $a_S = -2 $ \  and \ $a_T = -1 $ , }         
      \put(4.5,3.9){ $S=(-\frac{3}{2}|1) $ \  and \ $T=(-\frac{1}{2}|1) $ . }  
      \put(4.5,3.2){Hence \ $P_{hor}=(-\frac{5}{2}|1) $ \ and \ $P_{ver}=(\frac{3}{2}|1) $ . }  
      \linethickness{0.5mm} \put(-2.5,1.0){\line(1,0){2.0}}    \put(+1.5,1.0){\line(0,-1){4.0}}        
      \end{picture}
    \end{figure}	   
	      \\ \\   \\  {  }  \\      \\   \\  \\  {  }  
	   \begin{proof} 
	   Note that, if \ $G_S = G_T $ , the proposition is trivial. Hence we  assume   that 
	                              \ $G_S, G_T $ \ are distinct. \ \
	   We describe the parallel straight lines \ $G_S , G_T $ with equations 
	   $$ G_S \ := \ \{ ({\tt x}|{\tt y}) \in \mathbbm{R}^{2}\ | 
	                         \ e \cdot {\tt y} =  m \cdot {\tt x} + b_S \} \ \ \ { \rm and } \ \ \ 
	      G_T \ := \ \{ ({\tt x}|{\tt y}) \in \mathbbm{R}^{2}\\ | 
	                         \  e \cdot {\tt y} = m \cdot {\tt x} + b_T \} \ , $$
	   with \ \ $ e, m, b_S, b_T \in  \mathbbm{R}$ , \ $ ( m,e ) \neq (0,0) $ .  \
	   Without less of generality let either be \  \ 
	         ( $ e = 0 $ \ and \ $ m = 1$ ) \ or \ ( $ e = 1 $ ) .  \\
	 The straight line \ $L$ \ can be described  with two numbers \ $w_1, w_2 \in  \mathbbm{R}, \ (w_1,w_2) \neq (0,0)$ . 
	    $$ L \ := \  \{ \ ({\tt x}_S|{\tt y}_S ) + t \cdot ( w_1|w_2 ) \ | \ t \in  \mathbbm{R} \ \}     
	            = \  \{ \ ({\tt x}_T|{\tt y}_T ) + t \cdot ( w_1|w_2 ) \ | \ t \in  \mathbbm{R} \ \} \ , $$ 
	    with \ \
	   $ w_1 \cdot {\tt y}_T  \neq w_2 \cdot {\tt x}_T$ , \ and \ $ w_1 \cdot {\tt y}_S  \neq w_2 \cdot {\tt x}_S  $ ,
	    \ ( because $ (0|0) \notin L $ ) , \ and with \
	   $ e \cdot w_2 \neq m \cdot w_1 $, ( because $ L $ is not parallel  to $ G_S$ and $ G_T$ , respectively) .              \begin{lemma}  \label{Lemma 1}  \rm
	   In the case of   { \sf (A)}, (that means that \ $  G_S \ { \rm and} \ G_T $ \  are not parallel to the 
	   horizontal {\tt x}-axis), we have  \ $ m \neq 0 $ , \ and \ $ a_S = -b_S/m \ , \ a_T = -b_T/m $ .      
	   Then there are uniquely three numbers \ $ \varrho, \alpha, \beta \in  \mathbbm{R}$ \ which solve the
	   system of four linear equations 
	   $$    (1) \quad {\tt x}_S + \varrho \cdot w_1 - a_S = \alpha \cdot {\tt x}_T \quad ,  \qquad                                                (2) \quad {\tt y}_S + \varrho \cdot w_2   = \alpha \cdot {\tt y}_T \quad ,     $$
	   $$    (3) \quad {\tt x}_S + \varrho \cdot w_1 - a_T = \beta \cdot {\tt x}_S \quad ,  \qquad                                                 (4) \quad {\tt y}_S + \varrho \cdot w_2   = \beta \cdot {\tt y}_S \ \quad .   $$
	    In the case of   { \sf (B)}, (that means that \ $  G_S \ { \rm and} \ G_T $ \  are not parallel to the 
	    vertical {\tt y}-axis),  
	    there are uniquely three numbers \ $  \widetilde{\varrho},  \widetilde{\alpha}, 
	      \widetilde{\beta} \in  \mathbbm{R}$ \ which solve the   system of four equations 
	   $$    \widetilde{(1)} \quad 
	   {\tt y}_S + \widetilde{\varrho} \cdot w_2 - b_S = \widetilde{\alpha} \cdot {\tt y}_T \quad , \qquad  
	 \widetilde{(2)} \quad {\tt x}_S + \widetilde{\varrho} \cdot w_1   = \widetilde{\alpha} \cdot {\tt x}_T \quad ,  $$
	   $$    \widetilde{(3)} \quad 
	         {\tt y}_S + \widetilde{\varrho} \cdot w_2 - b_T = \widetilde{\beta} \cdot {\tt y}_S \quad , \qquad               \widetilde{(4)} \quad {\tt x}_S + \widetilde{\varrho} \cdot w_1  = \widetilde{\beta} \cdot {\tt x}_S \quad . $$
	   \end{lemma}
	 \begin{proof} \rm  
	 In case { \sf (A)} \  the two equations \ (1),(2) \ \ yield \ $ \varrho_{[1]} $ , \ 
	 and the two equations \ \ (3),(4) \ \ yield \ $ \varrho_{[2]} $ ,  
	 $$  \varrho_{[1]} = \frac{{\tt y}_S \cdot {\tt x}_T - {\tt x}_S \cdot {\tt y}_T + a_s \cdot {\tt y}_T}
	  { w_1 \cdot {\tt y}_T - w_2 \cdot {\tt x}_T }   \quad \ \  { \rm and } \quad  \ \ 
	           \varrho_{[2]} = \frac{{\tt y}_S \cdot a_T}{ w_1 \cdot {\tt y}_S - w_2 \cdot {\tt x}_S } \ . $$  
	  Because of \  $ ({\tt x}_S|{\tt y}_S), ({\tt x}_T|{\tt y}_T) \in L $, \ \ there is a 
	 $ \breve{t} \in \mathbbm{R}$ \ such that \ 
	 $ ({\tt x}_T|{\tt y}_T) = ({\tt x}_S|{\tt y}_S) + \breve{t} \cdot (w_1,w_2) $, hence \ \ 
	  $ { w_1 \cdot {\tt y}_T - w_2 \cdot {\tt x}_T } = { w_1 \cdot {\tt y}_S - w_2 \cdot {\tt x}_S } $ .
	And with \ $ ({\tt x}_S|{\tt y}_S) \in G_S \ , \   ({\tt x}_T|{\tt y}_T) \in G_T $ \  follows easily that \ \
	\ $ {{\tt y}_S \cdot {\tt x}_T - {\tt x}_S \cdot {\tt y}_T + a_s \cdot {\tt y}_T} = {{\tt y}_S \cdot a_T} $ \ ,  \
	\ hence \ \ \quad   $   \varrho_{[1]} = \varrho_{[2]} =: \varrho $ .  \\ 
	 In the case of { \sf (B)} \  the two equations \ $ \widetilde{(1)}\widetilde{(2)}$  \ \ yield \ 
	 $  \widetilde{\varrho}_{[1]} $ , \ \ and  \ \ $ \widetilde{(3)}\widetilde{(4)} $ \ \ yield \ 
	 $  \widetilde{\varrho}_{[2]} $ , 
  $$ \widetilde{\varrho}_{[1]}  = \frac{{\tt y}_T \cdot {\tt x}_S - {\tt x}_T \cdot {\tt y}_S + b_s \cdot {\tt x}_T}
	  { w_2 \cdot {\tt x}_T - w_1 \cdot {\tt y}_T }   \quad \ \  { \rm and } \quad  \ \ 
	      \widetilde{\varrho}_{[2]}  = \frac{{\tt x}_S \cdot b_T}{ w_2 \cdot {\tt x}_S - w_1 \cdot {\tt y}_S } \ . $$  
	 and  with similar steps as only just follows \ \quad 
	                      $  \widetilde{\varrho}_{[1]} = \widetilde{\varrho}_{[2]} =: \widetilde{\varrho} $ ,
	   and the lemma is proved.   
	  \end{proof}  
	 To finish the proof of proposition \ref{proposition Euclid1}   we set in the \\ case { \sf (A)}:  \quad \ 
	  $ P_{hor} \ = \ ({\tt x}_{hor}|{\tt y}_{hor}) \ := \ ({\tt x}_S|{\tt y}_S) + \varrho \cdot ( w_1|w_2) $ ,  \qquad 
	              and in the     \\  
	 	    case { \sf (B)}:  \quad \
	 $  P_{ver} \ = \ ({\tt x}_{ver}|{\tt y}_{ver}) \ := \ ({\tt x}_S|{\tt y}_S) + \widetilde{\varrho} \cdot ( w_1|w_2) $ . \\
	  The uniqueness of \ $  P_{hor} $ \ and \ $  P_{ver} $ \ is trivial, for instance, for a non vertical \ $L$, 
	   the horizontal distance (with signs) from a point on \ $ L $  \ to $ Z_S$ or $Z_T$ , respectively, 
   $$  {\tt x} \longmapsto { \rm the \ horizontal \ distance \ (with \ sign) \ from \ a \ point} \ \
    ({\tt x}|{\tt y}) \ \ { \rm on } \ L   \ { \rm to } \ Z_S  \quad { \rm and }    $$ 
   $$  {\tt x} \longmapsto { \rm the \ horizontal \ distance \ (with \ sign) \ from \ a \ point} \ \
    ({\tt x}|{\tt y}) \ \ { \rm on } \ L   \ { \rm to } \ Z_T \ , \quad  { \rm respectively } ,   $$   
   are  strictly  monotone functions.  This was the last what we had to do to prove the 
   proposition.          
	 
	   \end{proof}  
	  For completeness, we write down other representations of \ $  P_{hor} $ \ and \ $  P_{ver}$ , respectively. \\
	  Note that if  \  $ G_S , G_T $ \ are not parallel to the  vertical {\tt y}-axis, they have equations 
	    $$ G_S \ = \ \{ ({\tt x}|{\tt y}) \in \mathbbm{R}^{2}\ | 
	                         \ {\tt y} =  m \cdot {\tt x} + b_S \} \ \ \ { \rm and } \ \ \ 
	      G_T \ = \ \{ ({\tt x}|{\tt y}) \in \mathbbm{R}^{2} \ | 
	                         \ {\tt y} = m \cdot {\tt x} + b_T \} \ , $$ 
	    and  if  \  $ G_S , G_T $ \ are  vertical, they have equations   
	  $$ G_S \ = \ \{ ({\tt x}|{\tt y}) \in \mathbbm{R}^{2} \ | \ {\tt x} =  - b_S =: a_S \} \ \ \ { \rm and } \ \ \  
	     G_T \ = \ \{ ({\tt x}|{\tt y}) \in \mathbbm{R}^{2} \ | \ {\tt x} =  - b_T =: a_T \} \ . $$ 
	   If \ $L$ \ is not vertical, we have \ $ w_1 \neq 0 $ , and 
	    $$ L \ = \  \{ \ ({\tt x}_S|{\tt y}_S ) + t \cdot ( w_1|w_2 ) \ | \ t \in  \mathbbm{R} \ \}                               \ \  = \ \ \{ ({\tt x}|{\tt y}) \in \mathbbm{R}^{2}\ | \ {\tt y} =  m_L \cdot {\tt x} + b_L \} \ , $$
	    with \ \ $ m_L := w_2/w_1$ \ \ and \ \ $ b_L := {\tt y}_S  - {\tt x}_S \cdot m_L $ , \  
	                           \  ( $ b_L \neq  0 $ , \ because \ $ (0|0) \notin L $) . 
	    $$ { \rm  If } \ L \ { \rm is \ vertical, \ we \ set } \ \ a_L := {\tt x}_S = {\tt x}_T \ , { \ \rm and \ }
	     \quad   L \ \  = \ \ \{ ({\tt x}|{\tt y}) \in \mathbbm{R}^{2}\ | \ {\tt x} =  a_L  \} \ . $$   \\    
	    Now assume that neither \ $ L $ \ nor \ $ G_S, G_T $ \ are parallel to one of the axes. \ Then   
	  \begin{eqnarray*}   
        P_{hor} \ = \ \left( \ {\tt x}_{hor}|{\tt y}_{hor} \right) & = 
                     & \left( \ \frac{{\tt x}_T \cdot b_L + {\tt y}_T \cdot a_S}{b_L} 
     \ | \  m_L \cdot \frac{{\tt x}_T \cdot b_L + {\tt y}_T \cdot a_S}{b_L} + b_L \ \right)  \quad   \\   
      \qquad  & = &  \left( \ \frac{{\tt x}_S \cdot b_L + {\tt y}_S \cdot a_T}{b_L}
     \ | \   m_L \cdot \frac{{\tt x}_S \cdot b_L + {\tt y}_S \cdot a_T}{b_L} + b_L \ \right) \ \  = \\   
      \end{eqnarray*}     
      $$ \left( \frac{ {b_L}^{2} \cdot m + b_S \cdot b_T \cdot m_L - m \cdot b_L \cdot (b_S + b_T)}
                     {   b_L \cdot m \cdot (m - m_L) }    \ | \ 
     \frac{ {b_L}^{2} \cdot m^{2} + b_S \cdot b_T \cdot {m_L}^{2} -  m \cdot m_L \cdot b_L \cdot (b_S + b_T)  }
              {   b_L \cdot m \cdot (m - m_L) }  \right)  \ ,  $$ 
    \begin{eqnarray*}   
        P_{ver} \ = \ \left( \ {\tt x}_{ver}|{\tt y}_{ver} \right) & = & 
          \left( \ \frac{{\tt x}_T \cdot ( b_L - b_S) }{b_L} \ | \ 
                                            m_L \cdot \frac{{\tt x}_T \cdot ( b_L - b_S) }{b_L} + b_L \ \right) \\   
               & = &  \left( \ \frac{{\tt x}_S \cdot  ( b_L - b_T) }{b_L}
                                       \ | \   m_L \cdot \frac{{\tt x}_S \cdot ( b_L - b_T) }{b_L} + b_L \ \right) \\
               & = &  \left( \ \frac{ (b_L-b_T) \cdot  (b_L - b_S) }{b_L \cdot (m-m_L) }
                     \ | \ \frac{m_L \cdot ( b_S\cdot b_T -  b_L\cdot b_S -  b_L\cdot b_T ) + b_L^{\ 2}\cdot m}
                                                          {b_L \cdot (m-m_L) } \ \right) \ \ .                          
    \end{eqnarray*} 
   Now assume that \ $ G_S, G_T $ are not parallel to one of the axes, and $ L $ is horizontal.    
     ( See the previous picture, too.) \ \ Then we have an equation \ \ $ L: \ {\tt y}  = b_L $ , and  we get 
    \begin{eqnarray*}   
        P_{hor} \ = \ ( \ {\tt x}_{hor}|{\tt y}_{hor}) & = & ( \ {\tt x}_T +  a_S 
     \ | \ b_L \ ) \ = \ ( \ {\tt x}_S + a_T \ | \  b_L \ )  
                           \ = \ \left( \frac{b_L - b_S - b_T}{m} \ | \ b_L \right) \ ,  
    \end{eqnarray*}  
    \begin{eqnarray*}   
        P_{ver} & = & \left( \ \frac{{\tt x}_T \cdot ( b_L - b_S) }{b_L}  \ | \ b_L \ \right) \ = \
                         \left( \ \frac{{\tt x}_S \cdot  ( b_L - b_T) }{b_L} \ | \  b_L \ \right)  \ = \
                         \left( \ \frac{( b_L - b_S) \cdot  ( b_L - b_T) }{b_L \cdot m} \ | \  b_L \ \right)  \ .  
    \end{eqnarray*}    
   
    Assume that \ $ G_S, G_T $ are not parallel to one of the axes, and $ L $ is vertical.    
     \ \ Then we have an equation \ \ $ L: \ {\tt x}  = a_L $ , and  we get  
     \begin{eqnarray*}   
        P_{hor} \ = \ ( \ {\tt x}_{hor}|{\tt y}_{hor}) \ = \
          \left( \ a_L \ | \ m \cdot a_L + b_S + b_T + \frac{b_S \cdot b_T}{a_L \cdot m}  \ \right)  
        \ = \  \left( \ a_L \ | \ ( m + \frac{b_S}{a_L}\ ) \cdot  ( a_L + \frac{b_T}{m} ) \ \right) \ ,
      \end{eqnarray*}  
    $$  P_{ver} \ = \ \left( \ {\tt x}_{ver}|{\tt y}_{ver} \right) \ = \
                                      \left( a_L \ | \ m \cdot a_L + b_T + b_S \right) \ \ . $$  
    Now assume that \ $ G_S, G_T $ are  parallel to the horizontal {\tt x}-axis, and $ L $ is not parallel 
    to the  {\tt y}-axis (and, of course, not parallel to the {\tt x}-axis, too) . Then we get  no \ $  P_{hor} $  ,               and            \begin{eqnarray*}   
      P_{ver} \ = \ \left( \ {\tt x}_{ver}|{\tt y}_{ver} \right)  & = & 
                             \left( \ \frac{{\tt x}_T \cdot ( b_L - b_S) }{b_L} \ | \
                             m_L \cdot \frac{ {\tt x}_T \cdot ( b_L - b_S) }{b_L} + b_L \ \right) \\   
             & = &  \left( \ \frac{{\tt x}_S \cdot  ( b_L - b_T) }{b_L}
                        \ | \   m_L \cdot  \frac{ {\tt x}_S \cdot ( b_L - b_T) }{b_L} + b_L \ \right) \ \   \\
                & = &  \left( \ \frac{ (b_T-b_L) \cdot  (b_L - b_S) }{b_L \cdot m_L }
                     \ | \ \frac{ b_L\cdot b_S + b_L\cdot b_T -  b_S\cdot b_T }
                                                          {b_L } \ \right) \ .                   
    \end{eqnarray*} 
    If we assume that \ $ G_S, G_T $ are  parallel to the horizontal {\tt x}-axis, and $ L $ is parallel 
    to the  {\tt y}-axis, then we get no  \ $  P_{hor} $ , of course, and   
     \begin{eqnarray*}   
        P_{ver} \ = \ \left( \ {\tt x}_{ver} | {\tt y}_{ver} \right) & = & \left( \ a_L \ | \ b_S + b_T \ \right) \ .
     \end{eqnarray*}  
      Now assume that \ $ G_S, G_T $ are  parallel to the vertical {\tt y}-axis, and $ L $ is not parallel 
    to the  {\tt x}-axis (and, of course, not parallel to the {\tt y}-axis, too) .  Then we get  no \ $  P_{ver} $  ,         and    \begin{eqnarray*}   
        P_{hor} \ = \ \left( \ {\tt x}_{hor}|{\tt y}_{hor} \right) & = &  
                          \left( \ \frac{ a_T \cdot b_L + {\tt y}_T \cdot a_S}{b_L} 
                           \ | \  m_L \cdot \frac{ a_T \cdot b_L + {\tt y}_T \cdot a_S}{b_L} + b_L \ \right) \\
                & = &   \left( \ \frac{ a_S \cdot b_L + {\tt y}_S \cdot a_T}{b_L}
                    \ | \   m_L \cdot \frac{ a_S \cdot b_L + {\tt y}_S \cdot a_T}{b_L} + b_L \ \right)  \\                              & = &   \left( \ \frac{ b_L \cdot (a_S+a_T) + m_L \cdot a_S \cdot a_T }{b_L}
                      \ | \   m_L \cdot {\tt x}_{hor} + b_L \ \right)  \  .
    \end{eqnarray*} 
     And finally if we assume that \ $ G_S, G_T $ are  parallel to the  {\tt y}-axis, and $ L $ is parallel 
    to the  {\tt x}-axis, we get 
     \begin{eqnarray*}   
        P_{hor} \ = \  \left( \ {\tt x}_{hor}|{\tt y}_{hor} \right) & = &  \left( \ a_S +  a_T \ | \ b_L \ \right) \ .
     \end{eqnarray*}                  
    \begin{remark}    \rm  \quad    Note a few special trivial cases. \\  
   Assume that \ $ G_S, G_T $ are not  parallel to the horizontal {\tt x}-axis  (case{ \sf (A)}) .    \\
  If \ \ \ $ S = (a_S|0) $ , \ \ \
  then  we have \quad \ $ Z_S =  {\tt x}$-axis \ \ and \ \ $  P_{hor} = S = ( a_S | 0 ) $ .      \\
  If \ \ \ $ T = (a_T|0) $  , \ \ \
  then  we have \quad \ $ Z_T =  {\tt x}$-axis \ \ and \ \ $  P_{hor} = T = ( a_T | 0 ) $ .      \\
    Assume now that \ $ G_S, G_T $ are not  parallel to the vertical {\tt y}-axis  (case{ \sf (B)}) .   \\
  If \ \ \ $ S = (0|b_S) $ ,  \ \ \
  then  we have \quad \ $ Z_S =  {\tt y}$-axis \ \ and \ \ $  P_{ver} = S = ( 0 | b_S ) $ .      \\
  If \ \ \ $ T = (0|b_T) $ , \ \ \
  then  we have \quad \ $ Z_T =  {\tt y}$-axis \ \ and \ \ $  P_{ver} = T = ( 0 | b_T ) $ .
    \end{remark}       
    Now we describe  another proposition which seems  to be more general,  but indeed it is equivalent,
    see lemma  \ref{lemma Zwei}. \  
     Because we proved proposition  \ref{proposition Euclid1} , proposition  \ref{proposition Euclid2} 
     also is true.
    \begin{proposition}   \rm    \label{proposition Euclid2} 
    Let us take \ $ \mathbbm{R}^{2}$ ,
     \ the two-dimensional  euclidean plane. \
     Assume two parallel straight lines \ $  G_S \ { \rm and} \ G_T $. 
     Assume a third     line \ $ L $ , \ not parallel to \ $  G_S$ \ and \ $G_T $ , respectively.
     The  intersection of \ $ L $ \ with \ $ G_S $ \ is called \ $ S $, \ \ 
     and the intersection of \ $ L $ \ with \ $ G_T $ \ is called \  $ T $ . \
     Assume a fourth line \ $ Axis \ , \ \  Axis \neq L $ , \ and \ $ Axis$ \ is not parallel to \
      $  G_S$ \ and \ $G_T $ .
      The  intersection of \ $ Axis $ \ with \ $ G_S $ \ is called \ $ S_{Axis} $, \ \ 
     and the intersection of \ $ Axis $ \ with \ $ G_T $ \ is called \  $ T_{Axis} $ . \  
      Further we choose a point \ $ Origin $ \ on  \ $ Axis \backslash L $ . \ 
     We can draw  two  straight lines \ $ Z_S$  and  $Z_T $ , \ \ $ Z_S $ \ connects 
     \ $ Origin $ \  and \ $S$ , \ and \  $ Z_T $ \ connects \ $ Origin $ \  and $T$ .  \ 
      As every line, \ $ Z_S$ \ and \ $Z_T $ , respectively, \ divide the plane in two halfplanes. 
     \ \ $ Z_S$ \ and \ $Z_T $ \  are distinct if   \ $G_S$ \ and \ $ G_T $ \ are distinct. \ \
     \\  
       Then  there is an unique point \  $ { \sf P} \in L $ \ with the following properties: \\
     We draw the straight line \ $ Axis_{\:  { \sf P}} $  \ which meets \ 
  $   { \sf P} $ \ and which is parallel to \  $ Axis $ . \\
     We have to distinguish three cases, the main one (1) and two trivial ones (2),(3) .  \\
     (1) \qquad  \ $ S_{Axis} \neq S $ \ \ \ and \ \ \ $ T_{Axis} \neq T $ .  \\
     Then \ \ \ $ Axis_{\:   { \sf P}} \  \neq \  Axis $ ,  \
     the  intersection of \ $ Axis_{\:  { \sf P}}  $ \ with \ $ Z_S $ \ is called \ $ S_{ \sf P} $, \ \ 
     and the intersection of \ $ Axis_{\:  { \sf P}}  $ \ with \ $ Z_T $ \ is called \  $ T_{\sf P} $ . \\
     Then the distance of  \ $ S_{Axis} $ \  and  \ $ Origin $ \ is equal to the 
        distance of \ $ {\sf P} $ \ and  \ $ T_{\sf P} $, \  and 
       the distance of  \ $ T_{Axis} $ \  and  \ $ Origin $ \ is equal to the 
       distance of \ $ {\sf P} $ \ and  \ $ S_{\sf P} $ . \            
     Furthermore, \ \ \ $ S_{Axis} $ \ and \  $ {\sf P} $ \  are on the same side of \ $Z_T $ , \ and \ 
       $ T_{Axis} $ \ and \  $ {\sf P} $ \ are on the same side of \ $Z_S $ , respectively. \ ( See Picture 2 ) .  \\
      (2) \qquad  \  \ $ S_{Axis} = S $ .  \ ( Hence, \ if \ \ $ G_S \neq G_T $ \ then \ $ T_{Axis} \neq T $ ) .   \\
      Then  \ \ \  $ {\sf P} := S_{Axis} = S $ , \ \ and \ \  $ Axis_{\: {\sf P}} \  =  \  Axis \ = \ Z_S $ .  \\
      The intersection of \ $ Axis_{\: {\sf P}}  $ \ with \ $ Z_T $ \ is \ \ $ T_{\sf P} :=  Origin $ \  . \\
     Then, by  triviality,  the distance of \  $ S_{Axis} $ \  and  \ $ Origin $  \ is equal to the      
              distance of \ $ {\sf P} $ \ and  \  $ T_{\sf P} $, and furthermore,  by  triviality, \ \ \ 
        $ S_{Axis} $ \ and \  $ {\sf P} $ \  are on the same side of \ $Z_T $ .   \\
     (3) \qquad  \  \ $ T_{Axis} = T $ . \ ( Hence, \ if \ \ $ G_S \neq G_T $ \ then \ $ S_{Axis} \neq S $ ) .    \\
      Then  \ \ \ $ {\sf P} := T_{Axis} = T $ , \ \ and \ \  $ Axis_{\: {\sf P}} \  =  \  Axis \ = \ Z_T $  .  \\
     The  intersection of \ $ Axis_{\: {\sf P}}  $ \ with \ $ Z_S $ \ is  \ \ $ S_{\sf P} := Origin $ .  \\ 
     Then, by  triviality,  the distance of \ \ $ T_{Axis} $ \  and  \ $ Origin $ \ \ is equal to the      
       distance of \ $ {\sf P} $ \ and  \ $ S_{\sf P} $, and    
       furthermore,  by  triviality, \ \ \ $ T_{Axis} $ \ and \  $ {\sf P} $ \ are on the same side of \ $Z_S $ .  \\
     \end{proposition} 
     \setlength{\parindent}{0mm}
    \begin{figure}[ht]
    \centering
     \setlength{\unitlength}{0.9cm}
     \begin{picture}(6,6)     \thinlines  
      \put(8.11,-0,1){$  Axis    $}        \put(-0.1,-0.4){$ Origin $}   
   \put(-7.0,0){\line(1,0){15.0}}  
     \linethickness{0.5mm} \put(-6.0,0.0){\line(1,0){6.0}} 
   \thinlines          
   \put(-5.0,3.75){\line(1,0){15.0}}  \put(10.15,3.65){$ Axis_{\: {\sf P}} $} 
         \put(11.4,6.05){$G_T$}  \put(10.1,4.3){$G_S$} 
          \put(10.35,2.3){$ L $} 
         \put(-7.0,-0.33333333){\line(3,1){19.0}}  \put(-7.0,-1.33333333){\line(3,1){17.0}}                                                         \put(-3.0,4.5){\line(6,-1){14.0}}  
            \put(5.85,2.625){$S$}     \put(3.85,2.9){$T$}    
                \put(-4.0,-3.3333333){\line(6,5){11.5}} 
                \put(-4.0,-2.0000000){\line(2,1){14.5}} 
       \put(10.5,5.1){$ Z_S $}   
       \put(7.5,6.2){$ Z_T $}  
         \put(-3.0,-0.35){$  S_{Axis} $}        \put(-6.0,-0.35){$  T_{Axis} $}                
         \put(1.4,3.85){$ {\sf P} $}      \put(1.3,3.65){$\times$ } 
         \linethickness{0.5mm} \put(1.5,3.75){\line(1,0){6.0}}   
      \put(4.2,3.85){$ T_{\sf P} $}            \put(7.2,3.85){$ S_{\sf P} $}     
      \end{picture}
    \end{figure} 
        \begin{center}     Picture 2  \end{center}  
    \newpage
    \begin{lemma}   \label{lemma Zwei}  \rm   \qquad
    We have that \quad
     proposition \ref{proposition Euclid1} \ \ $\Longleftrightarrow $ \ \ proposition \ref{proposition Euclid2} .
    \end{lemma}
    \begin{proof}  \quad  \ \
     proposition \ref{proposition Euclid1} \ \ $ \Longleftarrow $ \ \ proposition \ref{proposition Euclid2}:  \\
     Obviously, the two situations which are described in \  proposition \ref{proposition Euclid1} are 
     special cases of the general situation in proposition \ref{proposition Euclid2} .  \ More detailed,
     we have \ \ $ Origin := ( 0|0 ) $ \  and \ 
     if we define \quad $ Axis $ := {\tt x}-axis  \ \ we get \ \ $ P_{hor} = {\sf P} $ , \ and  \ 
     \   $ Axis $  := {\tt y}-axis \ \ yields \ \ $ P_{ver} = {\sf P} $ , \ respectively.    \\
     proposition \ref{proposition Euclid1} \ \ $ \Longrightarrow $ \ \ proposition \ref{proposition Euclid2}:  \\
     With an easy transformation of coordinates, we  get \ $ ( 0|0 ) = Origin $ , \ and \ 
      {\tt x}-axis = $ Axis $ , hence \   $ {\sf P} = P_{hor}$ .  
    \end{proof} 
     \   \\ 
         Now follows another  piece of  'Grecian Geometry'.    
      \begin{proposition}   \rm     \label{proposition D} 
        Let us again take \ $ \mathbbm{R}^{2} = \{ ({\tt x}|{\tt y}) | \ {\tt x},{\tt y} \in \mathbbm{R} \} $ .
     with the  horizontal {\tt x}-axis and  the vertical  {\tt y}-axis.
     Consider the two parallel lines \ $ G, P $ ($P$ means 'projection line') , 
                                 \ with the property that $ G $ does not meet  $ (0|0) $.  
     Assume  a fixed \ $ \varepsilon \in  \mathbbm{R}$ . \  
     Let us choose  a point \ $ ( \widehat{{\tt x}} | \widehat{{\tt y}} )$ \ on $ G \ , \widehat{{\tt y}} \neq 0 $ , 
     such that  neither    the line that connects \ $ (0|0)$ \ and \   
     $ S := ( \widehat{{\tt x}}-\varepsilon  | \widehat{{\tt y}} )$, \ nor the line that connects \ $ (0|0)$ \ and \ 
      \ $ T := ( \widehat{{\tt x}}+\varepsilon   | \widehat{{\tt y}} )$ \ is parallel to \ $ G$ \ and \ $ P $ \ . 
     We call \ $ \overline{S}$ \ the projection of $ S $ on the line $ P $ , \  and  $ \overline{T}$ \ 
     the projection   of $ T $ on the line $ P $. \ (That means that the three points \
      $ (0|0), S , \overline{S}$, \ and the three points \ $ (0|0), T , \overline{T} $, respectively,
      \ are collinear, \ $ \overline{S}, \overline{T}  \in P $.)  The four points \ 
      $  \overline{S}, \overline{T},  -\overline{S}, -\overline{T} $ \ are the corners of a parallelogram.
      We call  \ $ '\nu' $ \ the intersection of the line that connects \ $ \overline{T} $ \ and \ $  -\overline{S}$ \  
       with the   horizontal  {\tt x}-axis. For the claim  we distinguish two disjoint cases:   \\
     { \sf (A)}  \quad If \ $ P $ \ and \  $ G $ \ 
         are parallel to the vertical {\tt y}-axis, then \ $ \nu $ \ depends only on  $ \varepsilon $ \ 
         and on the intersections of the horizontal \ {\tt x}-axis \ with $G$ \ and \ $P$, respectively. \\ 
      { \sf (B)} \quad If  $ P $ \ and \ $ G $ \ 
        are not parallel to the vertical {\tt y}-axis, then \ $ \nu $ \ depends  only  on  $ \varepsilon $ \ 
        and on the intersections of the vertical \ {\tt y}-axis \ with  $G$ \ and \ $P$,  respectively. 
         \ \   ( See Picture 3 ) .      \\  \\
     In other words we claim that the value of  \ $ \nu $ \
      does not depend on the choice of  \ $ ( \widehat{{\tt x}} | \widehat{{\tt y}} )$ \ on $ G $ ,   and also, \
      ( in case  { \sf (B)}  ) , \ that \ $ \nu $ \ does not depend on the slope  of \   $G$ \ and \ $P$ .     \\ 
      \end{proposition} \quad     \\   \\    
     \setlength{\parindent}{0mm}
    \begin{figure}[ht]
    \centering
     \setlength{\unitlength}{1cm}
    \begin{picture}(6,6)     \thinlines  
      \put(6.15,-0,1){$   {\tt x}   $}     \put(-0.3,5.8){$  {\tt y} $}    
       \put(-6.0,0){\vector(1,0){12.0}}     \put(0,-3){\vector(0,1){9}}  
       \put(-0.85,0.9){$ \overline{S} $}   \put(-4.0,4.1){$ S $}  \put(+0.85,-0.9){$ -\overline{S} $} 
        \put(-1.85,-2.4){$ \overline{T} $}   \put(4.1,3.7){$ T $}    \put(2.0,-0.3){$ \nu $} 
        \put(-0.9,4.2){$ ( \widehat{{\tt x}} | \widehat{{\tt y}} )$ }   
        \put(1.6,5.1){$P: {\tt y}=2\cdot{\tt x}+2$}  
           \put(1.0,6.1){$G: {\tt y}=2\cdot{\tt x}+4$}   \put(4.95,1.1){${\tt y}= \frac{1}{2} \cdot {\tt x}  -1$} 
             \put(-3.5,-3.0){\line(1,2){4.5}}  \put(-2.5,-3.0){\line(1,2){4.0}}   \put(-4.0,4.0){\line(1,0){8.0}} 
          \put(-5.0,5.0){\line(1,-1){7.0}}  \put(-3.5,-3.5){\line(1,1){9.0}}  \put(-3.0,-2.5){\line(2,1){8.0}} 
     \put(4.8,7.3){We fix two lines $ P,G$ (with slope 2) , } 
     \put(4.8,6.6){and  $  \varepsilon := 4$ . \ We choose \ $ (\widehat{{\tt x}} | \widehat{{\tt y}}) := (0|4)$ .}          \put(4.8,5.9){We  get $ \overline{S }=( -\frac{2}{3}|\frac{2}{3} ) ,  \quad \overline{T}=(-2|-2)$ . }    
     \put(6.8,5.2){ \ Hence we  get  $ \nu = 2 $ . }   
         \end{picture}
   \end{figure}	       
        \qquad   \\
          \hspace*{\fill}       Picture 3  \qquad  \qquad  
       \quad  \\   \\ 
      \begin{proof}
     First the trivial cases. If \ $ P $ \ meets the origin $(0|0)$ , then 
     $ (0|0) = \overline{T} = \overline{S} $ , and the parallelogram collapses into a single point $(0|0) =  \nu $ . 
      If \ $ \varepsilon = 0 $ , then   $ \overline{T} = \overline{S} $ , and the parallelogram degenerates 
      to a line between $ \overline{T}$ and $ -\overline{S} $ , \ that meets \ $(0|0) =  \nu $ . \\
     Hence we assume that \ $ P $ \ does not meet the origin $(0|0)$ , and we take (without loss of generality)
      an   \ $ \varepsilon > 0 $ . \quad We distinguish between vertical \ $ G,P$ \ and not vertical \ $ G,P$ . 
       Thus assume  vertical lines \ $ G $ \ and \ $ P$ \
        with equations \ $ G: \ {\tt x}=r $ \ and \ $  P: \ {\tt x}=p $ . 
      After choosing  a point \ $ ( \widehat{{\tt x}} | \widehat{{\tt y}} )$ \ on $ G \ , \widehat{{\tt y}} \neq 0 $ ,
      we can compute \ $  \overline{S} $ \ and \ $  \overline{T}$ .  \    
     Because  $ G $ does not meet  $ (0|0) $ \ we have \ $ r \neq 0 $ , \ and some easy
    calculations yield \ \ $ \nu = p \cdot \varepsilon / r $ .   \\
   In the case that  \ $ G,P$ \ are not vertical they have a  slope \ $m \in \mathbbm{R}$, \ and there are equations 
     \begin{equation*}
     G: \ {\tt y} = m \cdot {\tt x} + b_G \qquad { \rm and} \qquad   P: \ {\tt y} = m \cdot {\tt x} + b_P 
     \end{equation*}
     with \ \ $ m, b_G, b_P \in \mathbbm{R}, \ \  b_G, b_P \neq 0 $ .   
    After choosing a point \ $ ( \widehat{{\tt x}} | \widehat{\tt y} )$ \ on $ G \ , \ \widehat{{\tt y}} \neq 0 $ ,        we get with  elementary  calculations \  
     $$  \overline{S} = \frac{b_P}{b_G + m \cdot \varepsilon} \cdot 
    ( \widehat{{\tt x}} - \varepsilon \: | \: \widehat{{\tt y}} ) \qquad { \rm and } \qquad 
      \overline{T} = \frac{b_P}{b_G - m \cdot \varepsilon} \cdot 
    ( \widehat{{\tt x}} + \varepsilon \ | \ \widehat{{\tt y}} ) \ . $$ 
    Some more calculations yield the formula  
    $$ {\tt y} = \frac{ m \cdot \widehat{{\tt x}} + b_G } {\widehat{{\tt x}} \cdot b_G + m \cdot \varepsilon^{2}}
       \ \cdot \ [ \; b_G \cdot {\tt x} - b_P \cdot \varepsilon \; ]  $$ 
        for a non vertical straight line that intersects \
       $  \overline{T}  \   \rm  and \  -\overline{S} $ , \  and finally we get \quad
      $ \nu = b_P \cdot \varepsilon / b_G $ . \\ If the line that connects \
       $ \overline{T} \ \rm and \  -\overline{S} $ \ is vertical  we get the same formula for
       \ $ \nu $ , \ and  the proof of the   proposition is  complete.  
   \end{proof} 
  \begin{remark}    \rm
   The four points \ $ \overline{S}, \overline{T},  -\overline{S}, -\overline{T} $ \ form the corners 
   of a parallelogram, and, corresponding to the the value of $ \nu $ which is the intersection of
   the line through $ \overline{T} $ and $ -\overline{S}$ with the horizontal axis, the line through
     $ \overline{S} $ and $ -\overline{T}$ meets the same axis in   $-\nu  $, hence in $ -p \cdot \varepsilon / r $,
   (if both lines $ G,P $ are vertical ) , \ or in \ \  $ -b_P \cdot \varepsilon / b_G $       
    (if both lines $ G,P $ are not vertical ) .
   \end{remark} 
   \begin{corollary}    \rm
    If we reverse the roles  of  the   {\tt x}-axis  and {\tt y}-axis, we are able to formulate a 
    corresponding statement: 
      Consider the two parallel lines \ $ G, P $ , \ with the property that $ G $ does not meet  $ (0|0) $.  
     Assume  a fixed \ $ \varepsilon \in  \mathbbm{R}$ . \  
     Let us choose  a point \ $ ( \widehat{{\tt x}} | \widehat{{\tt y}} )$ \ on $ G \ , \widehat{{\tt x}} \neq 0 $ , 
     such that  neither    the line that connects \ $ (0|0)$ \ and \   
     $ S_v := ( \widehat{{\tt x}} | \widehat{{\tt y}}-\varepsilon  )$, \ nor the line that connects \ $ (0|0)$ \ and \ 
      \ $ T_v := ( \widehat{{\tt x}}  | \widehat{{\tt y}}+\varepsilon  )$ \ is parallel to \ $ G$ \ and \ $ P $ \ . 
     We call \ $ \overline{S_v}$ \ the projection of $ S_v $ on the line $ P $, \  and  $ \overline{T_v}$ \ 
     the projection   of $ T_v $ on the line $ P $. \ (That means that the three points \
      $ (0|0), S_v , \overline{S_v }$, \ and the three points \ $ (0|0), T_v  , \overline{T_v } $, respectively,
      \ are collinear, \ $ \overline{S_v }, \overline{T_v }  \in P $.)  The four points \ 
      $  \overline{S_v }, \overline{T_v },  -\overline{S_v }, -\overline{T_v } $ \ are the corners of a parallelogram.
      We call \ $ '\mu' $ \ the intersection of the line that connects \ $ \overline{T_v } $ \ and 
      \ $  -\overline{S_v }$ \   with the    vertical  {\tt y}-axis , 
      \  and we claim that the value of  \ $ \mu $ \  does not depend on the choice of  \ 
      $ ( \widehat{{\tt x}} | \widehat{{\tt y}} )$ \ on $ G $ .       
    \end{corollary}
     \begin{proof} Trivial with the previous proposition. 
        \end{proof}
       Again we  write down the last proposition in a seemingly more general form.
      \begin{proposition}   \rm      \label{proposition Euclid3} 
        Let us again take the euclidean space \     $ \mathbbm{R}^{2}$ . \
      Consider the two parallel lines \ $ G, P $ ($P$ means 'projection line') , \ and an arbitrary third line 
      ( $ \neq G $ ) \ that we 
      call \  $ Axis $  . We fix a point \ $ Origin$ \ on \ $ Axis \backslash G $ ,  
      and  an \ $ \varepsilon \geq 0 $ . \  
     Let us choose  a point \ $ ( \widehat{{\tt x}} | \widehat{{\tt y}} )$ \ on $ G \backslash Axis $, 
     and draw the straight line \  $  \widehat{Axis} $ , \ meeting \ $( \widehat{{\tt x}} | \widehat{{\tt y}} )$ \ 
     and parallel to \ $ Axis$  . \  We  mark two unique points \ $ S, T $ \ on \ $  \widehat{Axis} $ , such that the 
     distance of both to \ $ ( \widehat{{\tt x}} | \widehat{{\tt y}} )$ \  is \ $ \varepsilon $ .  
     We assume the extra property that the line that connects  \ $ Origin $ \ and \ $ S $ \ and the line
     that   connects  \ $ Origin $ \ and \ $ T $ \ are not parallel to \ $ G $ and $ P $ , respectively.  
     Thus we are able to   'project' \ $ S $ \ and \ $ T $ \ onto \ $ P $.   
      We call \ $ \overline{S} $ \ the projection of $ S $ on the line $ P $ , \  and \ $ \overline{T} $ \ 
     the projection   of $ T $ on the line $ P $, both projections relatively to $ Origin $ .
      \ (That means that the three points \
      $ Origin , S , \overline{S}$, \ and the three points \ $ Origin , T , \overline{T} $, respectively,
      \ are collinear, \ $ \overline{S}, \overline{T}  \in P $.) 
      Further we denote   two points \ $ -\overline{S},   -\overline{T} $,  \ such that the four points   
      $ Origin , S , \overline{S}, -\overline{S}$, \ and the four points \ $ Origin , T , \overline{T} ,                         -\overline{T}$, respectively, are collinear, and 
      the distance of \ $ Origin$ and $\overline{S} $ \ is equal to the distance of \ $ Origin$ and $-\overline{S}$,
       \     and \ 
    the distance of \ $ Origin$ and $\overline{T} $ \ is equal to the distance of \ 
    $ Origin$ and $-\overline{T}$, respectively. \
      The four points \ $  \overline{S}, \overline{T},  -\overline{S}, -\overline{T} $ \ form  a parallelogram 
      with centre \ $ Origin $ . \
      We call  \ $ '\nu' $ \ the intersection of the line that connects \ $ \overline{T} $ \ and \ $  -\overline{S}$ \  
       with   $ Axis $ ,    and we claim that   \ $ \nu $ \
       does not depend on the choice of  \ $ ( \widehat{{\tt x}} | \widehat{{\tt y}} )$ \ on $ G $ .  
           (See  Picture 4) . 
      \end{proposition} \quad  
      \setlength{\parindent}{0mm}
    \begin{figure}[ht]
    \centering
     \setlength{\unitlength}{0.9cm}
    \begin{picture}(6,6)     \thinlines  
      \put(6.15,-0,1){$ Axis $}   
       \put(-6.0,0){\line(1,0){12.0}}    
       \put(-0.85,0.9){$ \overline{S} $}   \put(-4.0,4.1){$ S $}  \put(+0.85,-0.9){$ -\overline{S} $} 
        \put(-1.88,-2.4){$ \overline{T} $}      \put(1.62,2.15){$ -\overline{T} $}   \put(1.85,1.9){$ \times $}  
         \put(4.0,3.6){$ T $}    \put(2.0,-0.3){$ \nu $} 
        \put(-0.8,4.1){$ ( \widehat{{\tt x}} | \widehat{{\tt y}} )$ }   
        \put(1.6,5.1){$P$}     \put(-0.5,-0.3){ $ Origin $ }
           \put(1.0,6.1){$G$}    
             \put(-3.5,-3.0){\line(1,2){4.5}}  \put(-2.5,-3.0){\line(1,2){4.0}} 
               \put(-6.0,4.0){\line(1,0){13.0}}     \put(7.1,3.9){$\widehat{Axis}$}
          \put(-5.0,5.0){\line(1,-1){8.0}}  \put(-3.5,-3.5){\line(1,1){9.0}}  \put(-3.0,-2.5){\line(2,1){8.0}} 
       \end{picture}
    \end{figure}	       
            \\ \\ \\   \\   \\  
      \quad    \hspace*{\fill}       Picture 4    \qquad \qquad       \qquad   \\
       \begin{lemma}    \rm   \qquad  \ \  
    We have that \quad
     proposition \ref{proposition D} \ \ $\Longrightarrow $ \ \ proposition \ref{proposition Euclid3} .
    \end{lemma}
    \begin{proof}  \quad  \ \
     With an easy transformation of coordinates, we  get \ $ ( 0|0 ) = Origin $ , \ and \ 
     {\tt x}-axis = $ Axis $ .  
     \end{proof}    
     \begin{remark}    \rm   \quad 
     For all propositions 
     it would be desirable to have a construction with compass and ruler, 
     using the classical methods of the 'old Greeks'. 
     \end{remark}    
     { \bf Acknowledgements } \\ We thank \ Dr. Gencho Skordev, Verena Thürey, Dr. Nils Thürey and specially 
      Prof. Dr. Eberhard  Oeljeklaus for interesting discussions, critical comments and some new ideas.  \\  \\ 
      \centerline{VOLKER  THÜREY  }   
      \centerline{Rheinstr. 91  }   
      \centerline{ 28199 Bremen,  Germany  } 
      \centerline{T: 49 (0)421/591777   } 
      \centerline{ E-Mail: volker@thuerey.de } 
    \quad   \\

	\end{document}